\documentclass[11pt, reqno]{amsart}
\usepackage{amsfonts,latexsym,amsthm,amssymb,graphicx,bbm,enumitem}
\usepackage[all]{xy}
\DeclareFontFamily{OT1}{rsfs}{}
\DeclareFontShape{OT1}{rsfs}{n}{it}{<-> rsfs10}{}
\DeclareMathAlphabet{\mathscr}{OT1}{rsfs}{n}{it}

\newtheorem{theorem}{Theorem}[section]

\newtheorem{corol}[theorem]{Corollary}
\newtheorem{prop}[theorem]{Proposition}

\newtheorem{question}[theorem]{Question}

{\theoremstyle{definition} \newtheorem{defin}[theorem]{Definition}}
{\theoremstyle{remark} \newtheorem{remark}[theorem]{Remark}
\newtheorem{example}[theorem]{Example}

\newcommand{\Abb}{{\mathbb{A}}}

\newcommand{\Pbb}{{\mathbb{P}}}

\newcommand{\Zbb}{{\mathbb{Z}}}

\newcommand\restr[2]{{
  \left.\kern-\nulldelimiterspace 
  #1 
  \vphantom{\big|} 
  \right|_{#2} 
  }}
  
\title[Chern Classes for Free Divisors with Linear Type Jacobian]{Chern Classes of Logarithmic Derivations for Free Divisors with Jacobian Ideal of Linear Type}

\begin{document}

\title[Chern Classes for Free Divisors with Linear Type Jacobian]{Chern Classes of Logarithmic Derivations for Free Divisors with Jacobian Ideal of Linear Type}

\author{Xia Liao}
\address{
KIAS,
85 Hoegiro, Dongdaemun-gu,
Seoul 130-722,
Republic of Korea
}
\email{liao@kias.re.kr}

\subjclass[2010]{Primary 14C17; Secondary 14J17}

\keywords{Chern-Schwartz-MacPherson class, Logarithmic derivation, Jacobian ideal of linear type}

\begin{abstract}
Let $X$ be a nonsingular variety defined over an algebraically closed field of characteristic $0$, and $D$ be a free divisor with Jacobian ideal of linear type. We compute the Chern class of the sheaf of logarithmic derivations along $D$ and compare it with the Chern-Schwartz-MacPherson class of the hypersurface complement. Our result establishes a conjecture by Aluffi raised in \cite{hyparr}.
\end{abstract}

\maketitle

\section{introduction}
Let $X$ be a $n$-dimensional nonsingular variety defined over an algebraically closed field of characteristic $0$, $D$ a reduced effective divisor and $U=X\smallsetminus D$ the hypersurface complement. The present paper is the third one in our sequence of studies on the following question: 
\begin{question}\label{q}
In the Chow group $A_*(X)$, under what conditions is the formula
\begin{equation}\label{formula}
c_{\textup{SM}}(\mathbbm{1}_U)=c(\textup{Der}_X(-\log D))\cap [X]
\end{equation}
true?
\end{question}

The left hand side of the formula is the Chern-Schwartz-MacPherson class of the open subvariety $U$, and the right hand side is the total Chern class of the sheaf of logarithmic derivations along $D$. 

Formula \eqref{formula} is previously known to be true in the following cases:
\begin{itemize}
\item $X$ is a nonsingular algebraic surface and $D$ is a locally quasi-homogeneous divisors \cite{liao1}. 
\item $X$ is a nonsingular variety and $D$ is a certain type of hypersurface arrangement \cite{aluffi}. This in particular includes the cases for free hyperplane arrangements in $\Pbb^n$ and simple normal crossing divisors in nonsingular varieties, which were treated individually in \cite{hyparr} and \cite{MR1697199}.
 \item When $X$ is a nonsingular projective complex variety and $D$ is a locally quasi-homogeneous free divisor, the classes in \eqref{formula} have the same images in the Chow group of the ambient projective space \cite{liao2}.
\end{itemize}

The main result of our investigation in this paper is:

\begin{theorem}\label{main}
Under the condition that $X$ is a nonsingular variety defined over an algebraically closed field $k$ of characteristic 0 and $D$ is a free divisor with Jacobian ideal of linear type, formula \eqref{formula} is true.
\end{theorem}

In view of the fact that locally quasi-homogeneous divisors are divisors with Jacobian ideal of linear type \cite{MR1931962}, the result of this paper not only covers all known cases in which formula \eqref{formula} is true, but also answers a conjecture raised by Aluffi in \cite{hyparr}. On the other hand, the fact that our result works for varieties defined over an arbitrary algebraically closed field suggests that our approach is purely algebraic, without resorting to local analytic geometry and the GAGA principle, which were applied in some previous works.

In the course of our study we realized a very interesting connection from question \ref{q} to the Logarithmic Comparison Theorem (LCT). A divisor is said to satisfy LCT if the natural morphism of complexes $\Omega_X^{\bullet}(\log D) \to \Omega_X^{\bullet}(\star D)$ is a quasi-isomorphism.
In \cite{MR1363009}, the authors proved that locally quasi-homogeneous free divisors satisfy LCT. Later, Calder\'{o}n Moreno and Narv\'{a}ez Macarro studied the LCT problem for integrable logarithmic connections with respect to a free divisor with Jacobian ideal of linear type \cite{MR2500865}. More recently, Narv\'{a}ez Macarro was able to show divisors with Jacobian ideal of linear type 
also satisfy LCT \cite{macarro}. In fact, in our previous approach to question \ref{q} \cite{liao2}, LCT was the key to allow us to compare the degrees of the classes in \eqref{formula}. Although in our current approach to question \ref{q} LCT is not utilized, we still wonder if any deeper relations between LCT and formula \eqref{formula} exist. We also note that there are several classes of divisors, such as Kozul free divisors and Euler homogeneous divisors that are highly relevant to LCT \cite{MR1931962}, and we ask if formula \eqref{formula} is still true for those types of divisors.

In the special case that $X = \Pbb^n$, the Chow group is a free abelian group of rank $n+1$ generated by classes of projective subspaces of $\Pbb^n$. Thus the classes in \eqref{formula} can be viewed as polynomials of degree $n$ whose degree $k$ terms are given by the codimension $k$ components of their corresponding classes. Assume moreover that $D$ is a hyperplane arrangement. In this case it is known by the work of Aluffi \cite{hyparr} that the CSM class polynomial is equivalent to the characteristic polynomial of the arrangement up to a change of variables. Combining Aluffi's formula, theorem \ref{main} generalizes a formula given by Musta{\c{t}\u{a}} and Schenck (Theorem 4.1 in \cite{MR1843320}). 

As we begin our main discussion, we will start from reviewing basic properties of Chern-Schwartz-MacPherson classes, especially how these classes arise from taking the shadows of Lagrangian cycles. Then we will see that the Lagrangian cycle corresponding to a free divisor with Jacobian ideal of linear types takes a very simple form, realized by basic operations in intersection theory. The main result will follow by combining these observations.    

\section*{Acknowledgement}
The result presented in this paper is a part of the my doctoral thesis in Florida State University. I am very grateful to my advisor Paolo Aluffi for his patient guidance, his deep mathematical insights. The support and encouragement I received from him can hardly be described by words.

\section{Chern-Schwartz-Macpherson classes as shadows of Lagrangian cycles}

A more detailed account of the materials presented here can be found in \cite{MR2097164} \cite{MR1063344}.

We work over an algebraically closed field $k$ of characteristic $0$. Let $X$ be an algebraic variety defined over $k$, and $W$ be a closed subvariety of $X$. The characteristic function $\mathbbm{1}_W$ of $W$ is the function on $X$ which takes the value $1$ on (closed) points of $W$ and $0$ elsewhere. A constructible function on $X$ is a $\Zbb$-linear combination
\begin{equation*}
\displaystyle \sum_W n_W \cdot \mathbbm{1}_W
\end{equation*}  
where the summation is taken over all closed subvarieties of $X$ and $n_W \in \Zbb$ are nonzero for only finite indexes. We denote by $C(X)$ the group of constructible functions on $X$. 

For a proper morphism $f: X \to Y$ of complex algebraic varieties we can define the push-forward of constructible functions $f_*: C(X) \to C(Y)$ by $\Zbb$-linear extension of the following formula:
\begin{equation*}
f_*(\mathbbm{1}_W)(p) = \chi(f^{-1}(p) \cap W).
\end{equation*}
Here $p$ is an arbitrary (closed) point on $Y$ and $\chi$ is the topological Euler characteristic. For the extension of the definition of push-forward to $k$-varieties, we refer the reader to \cite{MR1063344}.

With the notions introduced above, we get a covariant functor $\mathcal{C}$ from a subcategory of $k$-varieties to the category of abelian groups, assigning $C(X)$ to each variety $X$ and $f_*$ to a proper morphism $f: X \to Y$. 

The Chow functor $\mathcal{A}$ is another important functor from (a subcategory of) $k$-varieties to the category of abelian groups. To each $X \in \textup{Var}(k)$ we assign $A_*(X)$ the Chow group of $X$, which is the group of algebraic cycles on $X$ modulo rational equivalence. With a proper morphism $f: X\to Y$, there is also a well defined push-forword homomorphism $f_*: A(X) \to A(Y)$. For more details, see \cite{MR732620}.

The readers may wonder what relations do the functors $\mathcal{C}$ and $\mathcal{A}$ have. Indeed, Grothendieck conjectured and MacPherson proved that there exists a unique natural transformation (called MacPherson transformation) \cite{MR0361141}:

\begin{equation*}
c_*: \mathcal{C} \leadsto \mathcal{A}
\end{equation*}

\noindent with the normalization property that for a nonsingular variety $X$, the induced homomorphism
\begin{equation*}
C(X) \to A(X)
\end{equation*} 
takes $\mathbbm{1}_X$ to the total Chern class of the tangent bundle of $X$: 
\begin{equation*}
\mathbbm{1}_X \mapsto c(TX) \cap [X]. 
\end{equation*}

Note that the push-forward formula for constructible functions and the normalization property together determine uniquely the homomorphism $C(X) \to A(X)$ for an arbitrary variety $X$, since there is always a resolution of singularity $\pi:\tilde{X} \to X$ such that $\pi$ is a proper morphism. 

Now we want to describe a functor $\mathcal{L}$ which interpolate the functors $\mathcal{C}$ and $\mathcal{A}$. The functor $\mathcal{L}$ assigns to a variety $X$ the group $L(X)$ of {\em{Lagrangian cycles}} over $X$. For the purpose of this paper, we only consider varieties which assumes a closed imbedding to a nonsingular ambient space. Let $i: X\to M$ be such an imbedding and $T^{*}M$ be the cotangent bundle of $M$. Then $L(X)$ is isomorphic to the free abelian subgroup of the group of algebraic cycles in $\Pbb(T^{*}M|_X)$, generated by the cycles of projectivized conormal spaces $[\Pbb(T^{*}_WM)]$ of closed subvarieties $W \subset X$. For $x$ a nonsingular point of $W$, the fiber of $T^{*}_WM$ over $x$ consists of linear forms on $T_{x}M$ which vanish at $T_{x}W$. Thus the restriction of $T^{*}_WM$ to the regular subscheme of $W$ is the genuine conormal bundle, and $T^{*}_WM$ is the closure of this vector bundle in $T^{*}M|_X$. 

A proper morphism $f: X \to Y$ has the ability to push forward the Lagrangian cycles over $X$ to the ones over $Y$. There is a nice description of this mechanism by viewing the Lagrangian cycles as the projectivized conormal spaces. We leave this point to \cite{MR1063344}.

As we proposed, there are two consecutive natural transformations:
\begin{equation*}
\mathcal{C} \leadsto \mathcal{L} \leadsto \mathcal{A}.
\end{equation*}

To understand these natural transformations, fix an algebraic variety $X$. For each closed subvariety $W$, there exists a constructible function $\textup{Eu}_W$ which is called the local Euler obstruction of $W$. It is known that these local Euler obstructions form a basis of abelian group for $C(X)$. The homomorphism $C(X) \to L(X)$ sends the local Euler obstruction of $W$ to the projectivized conormal space of $W$ in $M$, with an appropriate sign twist.
\begin{equation*}
\textup{Eu}_W \to (-1)^{\textup{dim}W}[\Pbb(T^{*}_WM)]
\end{equation*}

The second natural transformation which produces rational equivalent classes in $X$ involves only standard operations in intersection theory. Let $\zeta$ be the universal quotient bundle of rank $m-1$ on $\Pbb(T^{*}M|_X)$ where $m=\textup{dim}M$, and $\pi$ be the projection $\Pbb(T^{*}M|_X) \to X$. Given a Lagrangian cycle $\alpha$, the homomorphism $L(X) \to A_*(X)$ makes the following assignment:
\begin{equation*}
\alpha \mapsto \pi_*(c(\zeta) \cap \alpha).
\end{equation*}

Aluffi calls this particular operation of producing a rational equivalent class in the Chow group of the base scheme from a cycle $\alpha$ in the projective bundle {\em taking the shadow of} $\alpha$ \cite{MR2097164}. Its name is derived for the reason that ``the shadow neglects some of the information carried by the object that casts it". In certain nice cases, the structure theorem of the Chow group of the projective bundle allows us to reconstruct a cycle from its shadow.

The composition of the natural transformations $\mathcal{C} \leadsto \mathcal{L}$ and $\mathcal{L}
\leadsto \mathcal{A}$ is not quite the MacPherson transformation. However it differs from the MacPherson transformation only by a dual. Let $\gamma \in A_*(X)$ be a rational equivalent class, express $\gamma$ as the summation of its $i$-dimensional component $\gamma_i$:
\begin{equation*}
\gamma = \sum \gamma_i.
\end{equation*}
Its dual $\breve{\gamma}$ is defined to be:
\begin{equation*}
\breve{\gamma} = \sum (-1)^i\gamma_i.
\end{equation*}

For instance, let $X$ be a $n$-dimensional variety and $E$ be a rank $n$ vector bundle over $X$. If $\gamma$ is the total Chern class of the vector bundle $E$:
\begin{equation*}
\gamma = c(E) \cap [X],
\end{equation*}
then
\begin{equation*}
\breve{\gamma} = (-1)^n c(E^{\vee}) \cap [X].
\end{equation*} 

Therefore taking the dual is an involution of the Chow group. The MacPherson transformation $c_*$ is the composition of the two natural transformations described above, followed by a dual of the Chow group. More specifically, given $f \in C(X)$, to calculate its CSM class, one needs to find subvarieties $W_i$ of $X$ and integers $n_i$ such that 
\begin{equation*}
f=\sum_i n_i \cdot \textup{Eu}_{W_i}.
\end{equation*}
After doing this, imbed $X$ in a nonsingular ambient variety $M$ and form the Lagragian cycle
\begin{equation*}
\sum_i n_i \cdot (-1)^{\textup{dim}W_i}[\Pbb(T^{*}_{W_i}M)]
\end{equation*}
in the projective bundle $\Pbb(T^{*}M|_X)$. Finally, taking the dual of the shadow of this cycle yields the CSM class of $f$ in $A_*(X)$. In this paper, when the underlying variety is clear from the context, we simply write $c_{\textup{SM}}(f)$ for this class instead of $c_*(X)(f)$.

Given an arbitrary constructible function $f$, the steps to find out its CSM class are conceptual but less practical, because it is in general difficult to determine the subvarieties $W_i$ and the integers $n_i$. Nonetheless, when $X$ is a hypersurface in a nonsingular variety $M$ and $f = \mathbbm{1}_X$ the characteristic function of $X$, we know more or less how to compute the Lagrangian cycle of $f$. 

\begin{theorem}[\cite{MR2097164}]\label{cycle}
The Lagrangian cycle corresponding to $\mathbbm{1}_X$ is $(-1)^{\textup{dim}X}[\textup{qBl}_YX]$, where $Y$ denotes the singular subscheme of $X$ and $\textup{qBl}_YX$ denotes the quasi-symmetric blow-up of $X$ along $Y$.
\end{theorem}

The singular subscheme $Y$ can be viewed either as a subscheme of $X$ or a subscheme of $M$. If $h$ is a local equation of $X$ in $M$, then the ideal sheaf of $Y$ is locally generated by all partial derivatives of $h$ from the former perspective, or by all partial derivatives of $h$ together with $h$ from the latter perspective. 

According to Aluffi \cite{MR2097164}, given a closed embedding $W \subset V$ of schemes, there is a spectrum of relevant `blow-up' algebras, each one of which corresponds to a closed imbedding $V$ into an ambient (not necessarily nonsingular) scheme $M$. In the special case $M = V$, the blow-up algebra corresponding to $V \subset V$ is the Rees algebra of the Ideal sheaf of $W$. The constructions are functorial in the sense that any morphism $M \to N$ of ambient schemes of $X$ gives rise to an epimorphism of the blow-up algebras (with a reversed arrow). The quasi-symmetric blow-up algebra of $V$ along $W$ is defined to be the inverse limit of this system of blow-up algebras. Aluffi also defines the quasi-symmetric blow-up $\textup{qBl}_WV$ to be the projective scheme associated with the quasi-symmetric blow-up algebra.

The object $\textup{qBl}_WV$ may seem intangible at the first sight, but it is not as difficult to grasp as one may feel. In fact, the inverse system of blow-up algebras stabilizes at any nonsingular $M$, and the quasi-symmetric blow-up can be captured by the notion of `principal transform'. 

\begin{defin}[\cite{MR2097164}]
Let $W \subset V \subset M$ be closed imbeddings of schemes, with $M$ possibly singular. The principal transform of $V$ in the blow-up $\textup{Bl}_WM \xrightarrow{\rho} M$ of $M$ along $W$ is the residual to the exceptional divisor in $\rho^{-1}(V)$.
\end{defin}

The definition of residual subscheme can be found in Fulton's book \cite{MR732620} chapter 9. Its intuitive meaning in our context is the subscheme of $\textup{Bl}_WM$ obtained by subtracting one copy of the exceptional divisor from the total transformation of $V$. 

The principal transform and the quasi-symmetric blow-up are related in the following way:
\begin{theorem}[\cite{MR2097164}]\label{princ}
Let $W \subset V \subset M$ be closed imbeddings of schemes, with $M$ a nonsingular variety. The quasi-symmetric blow-up $\textup{qBl}_WM$ of $V$ along $W$ equals the principal transform of $V$ in $\textup{Bl}_WM$.
\end{theorem}

Therefore, the principal transform of $V$ in the blow-up $\textup{Bl}_WM$ as an abstract scheme, is independent of the ambient nonsingular variety $M$ by which we realize the principal transform.

Returning to the setup of Theorem \ref{cycle}, as a direct consequence of the previous discussion, we have:
\begin{theorem}[\cite{MR2097164}]\label{shadow}
The shadow of $[\textup{qBl}_YX]$ is $(-1)^{\textup{dim}X}\breve{c}_{\textup{SM}}(\mathbbm{1}_X)$.
\end{theorem}

In the next section, we will see that under the setup of Theorem \ref{cycle}, the quasi-symmetric blow-up of $X$ along $Y$ (which is also the principal transform of $X$ in the blow-up $\textup{Bl}_YM$) arises from the symmetric algebra of an ideal sheaf. We will compute this quasi-symmetric blow-up explicitly from a fundamental exact sequence in the theory of logarithmic derivations, and study how it is imbedded in the projectivized cotangent bundle of $M$.

\section{The imbedding of the quasi-symmetric blow-up in the cotangent bundle}

In this section and the following, $X$ is a nonsingular algebraic variety defined over an algebraically closed field of characteristic 0, $D$ is a reduced effective divisor (hypersurface) in $X$, $D_{sing}$ is the singular subscheme of $D$. Let $i: D \to X$ and $j: D_{sing} \to X$ be the inclusions of $D$ and $D_{sing}$ into $X$ respectively. For readers who follow the discussion of the previous section, there is a caution in the change of notation. The schemes $X$, $D$ and $D_{sing}$ play the role of $M$, $X$ and $Y$ in theorem \ref{cycle} respectively. The ideal sheaf of $D_{sing}$ in $\mathscr{O}_D$ will be denoted by $\mathscr{I}$, and the ideal sheaf of $D_{sing}$ in $\mathscr{O}_X$ will be denoted by $\tilde{\mathscr{I}}$.

The sheaf of logarithmic derivations $\textup{Der}_X(-\log D)$ along $D$ is a subsheaf of the sheaf of regular derivations $\textup{Der}_X$, which we also identify with the sheaf of sections of the tangent bundle of $X$. Over an open subset $U$ of $X$ where the divisor $D$ has a local equation $h$,
\begin{equation*}
\textup{Der}_X(-\log D)(U) = \{ \theta \in \textup{Der}_X(U) \ | \ \theta h \in h \cdot \mathscr{O}_X(U) \}.
\end{equation*}

\begin{example}
Let $X = \Abb^2$ and the ideal of $D$ be $(xy)$. The module of logarithmic derivations $\textup{Der}_X(-\log D)$ is free and it is generated by $x\partial_x, y\partial_y$.
\end{example}

The sheaf of logarithmic derivation $\textup{Der}_X(-\log D)$ is a reflexive sheaf for any reduced effective divisor $D$. Its dual is the sheaf of logarithmic differential $1$-forms $\Omega_X^{1}(\log D)$. One can also define the sheaf of logarithmic differential $k$-forms $\Omega_X^{k}(\log D)$ and form a logarithmic De Rham complex $\Omega_X^{\bullet}(\log D)$. For a discussion on properties of these sheaves at the introductory level, we refer the readers to \cite{MR586450}.

In this paper, the most significant aspect of the sheaf of logarithmic derivations we will use is that it fits into a fundamental exact sequence of sheaves of $\mathscr{O}_X$-modules:
\begin{equation}\label{funda}
0 \to \textup{Der}_X(-\log D) \to \textup{Der}_X \to i_*\mathscr{I}(D) \to 0
\end{equation}

In addition, the sheaf $\mathscr{I}$ and $\tilde{\mathscr{I}}$ are related by the following exact sequence:
\begin{equation}\label{jacob}
0 \to \mathscr{O}_X(-D) \to \tilde{\mathscr{I}} \to i_*\mathscr{I} \to 0.
\end{equation}

The morphism $\mathscr{O}_X(-D) \to \tilde{\mathscr{I}}$ is the inclusion of the equation of $D$ into the ideal of $D^{s}$.

\begin{remark}
The sheaf $i_*\mathscr{I}$ is an $\mathscr{O}_X$-module, but not an $\mathscr{O}_X$-ideal. In fact, applying the snake lemma to the following morphism of exact sequences

\centerline{\xymatrix{
0 \ar[r] & \mathscr{O}_X(-D) \ar[d] \ar[r] & \mathscr{O}_X \ar[d] \ar[r] & i_*\mathscr{O}_{D} \ar[d] \ar[r] & 0 \\
0 \ar[r] & \tilde{\mathscr{I}} \ar[r]             & \mathscr{O}_X  \ar[r]         & j_*\mathscr{O}_{D_{sing}} \ar[r]       & 0    } }

shows that $i_*\mathscr{I}$ is isomorphic to the kernel of the map $i_*\mathscr{O}_{D} \to j_*\mathscr{O}_{D_{sing}}$.
\end{remark}

To understand exact sequence \eqref{funda}, we consider as before an open subset $U$ over which the divisor $D$ has a local equation $h$. The morphism 
\begin{equation*}
\restr{\textup{Der}_X(-D)}{U} \to \restr{i_*\mathscr{O}_D}{U}
\end{equation*}
is defined by
\begin{equation*}
\theta \otimes g \mapsto g\cdot \overline{\theta h}
\end{equation*}
where $\theta \in \textup{Der}_X(U)$ and $g \in \Gamma(U, \mathscr{O}_X)$ which is treated as a section over $U$ of the line bundle $\mathscr{O}_X(-D)$ through the trivialization $\restr{\mathscr{O}_X}{U} \cong \restr{\mathscr{O}_X(-D)}{U}$. These locally defined morphisms glue together to give a morphism $\textup{Der}_X(-D) \to i_*\mathscr{O}_D$ whose image is generated by all partial derivatives of the local equation of $D$. Through this description we also see that $\theta \otimes 1$ maps to $0$ if and only if $\overline{\theta h} = 0$ in $\Gamma (U,\mathscr{O}_D)$, which is equivalent to $\theta h$ belonging to the ideal generated by $h$ in $\Gamma (U, \mathscr{O}_X)$. The latter condition is also equivalent to $\theta$ being a logarithmic derivation.

\begin{remark}
Restrict the morphism $\textup{Der}_X(-D) \to i_*\mathscr{O}_D$ to $D$ and dualize it, we get a morphism $\mathscr{O}_D \to \restr{\Omega_X^{1}(D)}{D}$. In \cite{MR2097164} \textsection 3.7, this morphism is constructed by considering an exact sequence which involves the bundle of principal parts of $\mathscr{O}(D)$. Later in this paper the morphism $\textup{Der}_X \to i_*\mathscr{I}(D)$ will be used to construct an imbedding of the quasi-symmetric blow-up $\textup{qBl}_{D_{sing}}D$ in the cotangent bundle of $X$. This imbedding is the same as the one appeared in \cite{MR2097164} \textsection 3.7.
\end{remark}

Before continuing our discussion, we first recall the definitions of free divisors and divisors whose Jacobian ideal is of linear type.

\begin{defin}
A reduce effective divisor is free if $\textup{Der}_X(-\log D)$ is locally free (of rank equal to the dimension of $X$) \cite{MR586450}. A divisor is Jacobian ideal of linear type if $\tilde{\mathscr{I}}$ is an ideal sheaf of linear type, which is saying $\textup{Rees}_{\mathscr{O}_X}(\tilde{\mathscr{I}}) \cong \textup{Sym}_{\mathscr{O}_X}(\tilde{\mathscr{I}})$ \cite{MR2500865}. 
\end{defin}

We also need a classical result concerning the symmetric algebras of modules.

\begin{theorem}[\cite{MR1322960} Appendix A.2.3]\label{sym}
Let $A$ be a commutative ring and 
\begin{equation*}
0 \to K \to M \to N \to 0
\end{equation*}
an exact sequence of $A$-modules. Then the induced homomorphism
\begin{equation*}
\textup{Sym}_A(M) \to \textup{Sym}_A(N)
\end{equation*} 
is surjective and the kernel of this homomorphism is generated by the image of $K$ in degree $1$ of the graded algebra $\textup{Sym}_A(M)$.
\end{theorem} 

We will apply a sheafified version of this theorem on exact sequence \eqref{funda} and \eqref{jacob} respectively.

\begin{prop}\label{quasi}
Let $D$ be a divisor with Jacobian ideal of linear type, we have
\begin{equation*}
\textup{Proj}\Big(\textup{Sym}_{\mathscr{O}_X}(i_*\mathscr{I})\Big) \cong \textup{qBl}_{D_{sing}}{D}.
\end{equation*}
\end{prop}

\begin{proof}
Applying theorem \ref{sym} to exact sequence \eqref{jacob}, and using the condition that $D$ is Jacobian of linear type, we get an epimorphism:
\begin{equation*}
\textup{Rees}_{\mathscr{O}_X}(\tilde{\mathscr{I}}) \to \textup{Sym}_{\mathscr{O}_X}(i_*\mathscr{I})
\end{equation*}
whose kernel is generated by $\mathscr{O}_X(-D)$ in degree $1$ of $\textup{Rees}_{\mathscr{O}_X}(\tilde{\mathscr{I}})$. Let us examine the kernel locally. Over an affine open set $U = \textup{Spec}(A)$ of $X$, assuming $h$ is the equation of $D$ in $U$, the Rees algebra takes the form
\begin{equation*}
\textup{Rees}_A(\tilde{I}) = A \oplus \tilde{I}t \oplus \tilde{I}^2t^2 \oplus \cdots
\end{equation*}
and the ideal defining $\textup{Sym}_A(I)$ is generated by $ht$. Consider a nonzero element $g$ in $\tilde{I}$. This element determines an open subset $D_{+}(gt)$ in $\textup{Proj}\Big(\textup{Rees}_A(\tilde{I})\Big)$, in which the ideal of $\textup{Proj}\Big(\textup{Sym}_A(I)\Big)$ is generated by 
\begin{equation*}
\frac{ht}{gt} = \frac{h}{g}.
\end{equation*}
In the open subset $D_+(gt)$, the ideal of the total transformation of $D$ is generated by $h$ and the ideal of the exceptional divisor is generated by $g$. Consequently the fraction $h/g$ locally defines the ideal of the principal transform of $D$ in the blow-up $\textup{Proj}\Big(\textup{Rees}_{\mathscr{O}_X}(\tilde{\mathscr{I}})\Big)$, which agrees with the quasi-symmetric blow-up of $D$ along $D_{sing}$ according to theorem \ref{princ}. 
\end{proof}

We want to introduce some auxiliary notation at this moment in order to smoothen the following discussion. From now on:
\begin{itemize}
\item $E$ will denote the cotangent bundle $\textup{Spec}\Big(\textup{Sym}_{\mathscr{O}_X}(\textup{Der}_X)\Big)$ of $X$, 
\item $F$ will denote the logarithmic cotangent cone $\textup{Spec}\Big(\textup{Sym}_{\mathscr{O}_X}\big(\textup{Der}_X(-\log D)\big)\Big)$ of $X$,
\item $C$ will denote the subcone $\textup{Spec}\Big(\textup{Sym}_{\mathscr{O}_X}\big((i_*\mathscr{I}(D)\big)\Big)$ of $E$,
\item $\sigma: E \to F$ will be the morphism induced by $\textup{Der}_X(-\log D) \to \textup{Der}_X$,
\item The corresponding projectivized cones and bundles of $E$, $F$, and $C$ will be denoted by $\mathbf{P}(E)$, $\mathbf{P}(F)$, and $\mathbf{P}(C)$. Notice that $\mathbf{P}(C)$ is the quasi-symmetric blow-up of $D$ along $D_{sing}$ because $\textup{Proj}\Big(\textup{Rees}_{\mathscr{O}_X}(i_*\mathscr{I})\Big) \cong \textup{Proj}\Big(\textup{Rees}_{\mathscr{O}_X}\big(i_*\mathscr{I}(D)\big)\Big)$.
\end{itemize}

Next, we apply theorem \ref{sym} to exact sequence \eqref{funda}.

\begin{prop}
The inverse image of the zero section of $F$ under $\sigma$ is $C$.
\end{prop}
\begin{proof}
This is immediate. The ideal of the zero section of $F$ is generated by $\textup{Der}_X(-\log D)$ in $\textup{Sym}_{\mathscr{O}_X}(\textup{Der}_X (-\log D))$. Thus the ideal of the inverse image of the zero section of $F$ in $E$ is generated by the image of $\textup{Der}_X(-\log D)$ in $\textup{Sym}_{\mathscr{O}_X}(\textup{Der}_X)$. This ideal also defines $\textup{Sym}_{\mathscr{O}_X}((i_*\mathscr{I}(D))$ according to theorem \ref{sym}.
\end{proof}

In the rest of the paper we will focus on free divisors with Jacobian ideal of linear type. Under this condition $F$ becomes a vector bundle of rank $n=\textup{dim}X$. The morphism $\sigma$ becomes a linear map between vector bundles. The cone $C$ can be thought of as the fiberwise kernel of $\sigma$.

\begin{example}
Again consider the case that $X=\mathbb{A}^2$ and $D$ is the normal crossing divisor defined by the ideal $(xy)$. The vector bundles $E$ and $F$ are both trivial of rank $2$. The homomorphism on the affine coordinate rings is:
\begin{equation*}
\begin{split}
k[x,y][A,B] &\to k[x,y][A,B] \\
A &\mapsto xA \\
B &\mapsto yB.
\end{split}
\end{equation*}
Let $v = v_1\mathbf{e_1} + v_2\mathbf{e_2}$ be a vector in $E$ over the point $(a,b) \in \mathbb{A}^2$. Then its image in $F$ by $\sigma$ is the vector $av_1\mathbf{e_1} + bv_2\mathbf{e_2}$ over the same point $(a,b)$. From this description we see that $C$ has a $2$-dimensional fiber over the point $(0,0)$; $1$-dimensional fibers generated by $\mathbf{e_2}$ over points on the $x$-axis; $1$-dimensional fibers generated by $\mathbf{e_1}$ over points on the $y$-axis; and $0$-dimensional fibers elsewhere. The projectivized cone $\mathbf{P}(C)$ has one copy of $\mathbb{P}^1$ over the point $(0,0)$, and two separate lines meeting $\mathbb{P}^1$ at two distinct points. It is the same as the principal transform of $D$ along its singular subscheme $(0,0)$, as the exceptional divisor of the blow-up of $D$ along the origin contains two copies of $\mathbb{P}^1$.
\end{example}

With these preparations, we can give a rather transparent description of the imbedding of $\mathbf{P}(C)$ in $\mathbf{P}(E)$.

\begin{prop}\label{normal}
Under the condition that $D$ is a free divisor with linear type Jacobian ideal, $\mathbf{P}(C)$ is a locally complete intersection in $\mathbf{P}(E)$. The normal bundle of $\mathbf{P}(C)$ in $\mathbf{P}(E)$ is isomorphic to $p^*F \otimes \mathscr{O}(1)$ where $p$ is the projection $\mathbf{P}(C) \to X$.
\end{prop}

\begin{proof}
Consider the following Cartesian square

\begin{displaymath}
\xymatrix{
&C \ar[d] \ar[r] & E \ar[d] \\
&\Gamma(\sigma) \ar[r] & E \oplus F}
\end{displaymath}
In the diagram $\Gamma(\sigma)$ is the graph of $\sigma: E \to F$. Set-theoretically and fiberwisely it consists of vectors of the form $(v, \sigma(v))$. The map $E \to E \oplus F$ is the inclusion $v \mapsto (v,0)$.  

Projectivize the above diagram we get another Cartesian square

\begin{displaymath}
\xymatrix{
&\mathbf{P}(C) \ar[d]^{\phi} \ar[r] & \mathbf{P}(E) \ar[d] \\
&\mathbf{P}(\Gamma(\sigma)) \ar[r] & \mathbf{P}(E \oplus F)}
\end{displaymath}

The projectivized cotangent bundle $\mathbf{P}(E)$ is irreducible and its dimension is $2n-1$. The quasi-symmetric blow-up $\mathbf{P}(C)$ of $D$ along $D_{sing}$ has pure dimension $n-1$, since we have seen in the proof of proposition \ref{quasi} the quasi-symmetric blow-up is a Cartier divisor in $\textup{Bl}_{D_{sing}}X$ with a local equation $h/g$. Thus the codimension of $\mathbf{P}(C)$ in $\mathbf{P}(E)$ equals $n$. The ideal of $\mathbf{P}(C)$ in $\mathbf{P}(E)$ is also locally generated by $n$ elements, because the rank of $\textup{Der}_X(-\log D)$ is $n$. From these we deduce that $\mathbf{P}(C)$ is a locally complete intersection scheme in $\mathbf{P}(E)$---a well known result for Cohen-Macaulay rings (\cite{MR1011461} theorem 17.4). 

Denote by $\tilde{q}$ the projection $\Gamma(\sigma) \to X$ and $q$ the projection $\mathbf{P}(\Gamma(\sigma)) \to X$. The locally complete intersection $\mathbf{P}(C)$ has a normal bundle of rank $n$ in $\mathbf{P}(E)$, and this normal bundle is a subbundle of $\phi^*N$ where $N$ is the normal bundle to $\mathbf{P}(\Gamma(\sigma))$ in $\mathbf{P}(E \oplus F)$ (\cite{MR732620} chapter 6, page 93). We know that $N$ also has rank $n$. As a result, the normal bundle to $\mathbf{P}(C)$ in $\mathbf{P}(E)$ must agree with $\phi^*N$. 

The closed imbedding $\mathbf{P}(\Gamma(\sigma)) \to \mathbf{P}(E \oplus F)$ falls along the standard situation that one projectivized vector bundle being imbedded into another. The normal bundle to $\Gamma(\sigma)$ in $E \oplus F$ is $\tilde{q}^*F$ (\cite{MR732620} appendix B.7.3), and thus $N \cong q^*F \otimes \mathscr{O}(1)$. The last statement can be seen by looking at the Euler sequences defining the tangent bundles of $\mathbf{P}(\Gamma(\sigma))$ and $\mathbf{P}(E \oplus F)$. 

Finally we observe that 
\begin{equation*}
\phi^*(q^*F \otimes \mathscr{O}(1)) = p^*F \otimes \mathscr{O}(1)
\end{equation*}
\end{proof}

\begin{remark}
Similarly, the normal bundle to $\mathbf{P}(E)$ in $\mathbf{P}(E \oplus F)$ is $r^*F \otimes \mathscr{O}(1)$, with $r$ the projection $\mathbf{P}(E) \to X$.
\end{remark}

\begin{corol}\label{imbed}
In $A_*(\mathbf{P}(E))$, we have $[\mathbf{P}(C)] = c_n(r^*F \otimes \mathscr{O}(1)) \cap [\mathbf{P}(E)]$.
\end{corol}

\begin{proof}
For any $t$ in the algebraically closed base field, there is a map $t\sigma: E \to F$ defined by $v \mapsto t \cdot \sigma(v)$. Thus there is a family of cycles in $\mathbf{P}(E \oplus F)$ deforming $\mathbf{P}(\Gamma(\sigma))$ $(t=1$) to $\mathbf{P}(E)$ $(t=0)$. Moreover, the intersection product of $[\mathbf{P}(\Gamma(t\sigma))]$ and $[\mathbf{P}(E)]$ is always $[\mathbf{P}(C)]$ for $t \neq 0$. Therefore:
\begin{equation*}
\begin{split}
[\mathbf{P}(C)] &= [\mathbf{P}(\Gamma(t\sigma))] \cdot [\mathbf{P}(E)] \quad (t \neq 0) \\
           &= \lim_{t \to 0} \ \big([\mathbf{P}(\Gamma(t\sigma))] \cdot [\mathbf{P}(E)]\big) \\
           &= [\mathbf{P}(E)] \cdot [\mathbf{P}(E)] \\
           &= c_n(r^*F \otimes \mathscr{O}(1)) \cap [\mathbf{P}(E)]
\end{split}
\end{equation*}

For the third equality we use the dynamic interpretation of the intersection product, and for the last equality we use the self intersection formula (\cite{MR732620} theorem 6.2).

\end{proof}

To summarize, in this section, we have realized the quasi-symmetric blow-up $\mathbf{P}(C)$ of $D$ along $D_{sing}$ concretely in the projectivized cotangent bundle $\mathbf{P}(E)$ of $X$ (proposition \ref{normal} and corollary \ref{imbed}). Corollary \ref{imbed} will be used to calculate the shadow of $\mathbf{P}(C)$ in $A_*(X)$, which is an essential step in obtaining the main theorem of this paper.

\section{Proof of the main theorem}

Engaging all elements we saw in the previous sections, the proof of Theorem \ref{main} is almost at hand.

\begin{proof}[Proof of theorem \ref{main}]
We rewrite the formula we want to prove as:
\begin{equation*}
c_{\textup{SM}}(\mathbbm{1}_X) - c(\textup{Der}_X(-\log D))\cap [X] = c_{\textup{SM}}(\mathbbm{1}_D)
\end{equation*}

Taking the dual of this formula and recalling that for a nonsingular variety $X$, $c_{\textup{SM}}(\mathbbm{1}_X) = c(TX) \cap [X]$, it is equivalent to verify that:
\begin{equation*}
(-1)^{n}c(T^*X)\cap [X] - (-1)^{n}c(\Omega^{1}_X(\log D))\cap [X] = \breve{c}_{\textup{SM}}(\mathbbm{1}_D).
\end{equation*}

If we invoke our notations in the previous section, we can again rewrite the formula as:
\begin{equation*}
c(E) \cap [X] - c(F) \cap [X] = (-1)^n \breve{c}_{\textup{SM}}(\mathbbm{1}_D).
\end{equation*}

By theorem \ref{shadow}, proposition \ref{quasi} and our discussion about the natural transformation $\mathcal{L} \leadsto  \mathcal{A}$, we get:
\begin{equation*}
\begin{split}
(-1)^{n-1}\breve{c}_{\textup{SM}}(\mathbbm{1}_D) &= r_{*}(c(\zeta) \cap [\mathbf{P}(C)]) \\
                                                                                           &= c(E) \cap r_{*}\Big(\big(c(\mathscr{O}_{\mathbf{P}(E)}(-1))\big)^{-1} \cap[\mathbf{P}(C)]\Big).
\end{split}
\end{equation*}

Here $r$ is the projection $\mathbf{P}(E) \to X$, and $\zeta$ is the universal quotient bundle of $\mathbf{P}(E)$. In getting the second equality, we use the projection formula and the Whitney sum formula:
\begin{equation*}
c(r^{*}E) = c(\zeta) \cdot c(\mathscr{O}_{\mathbf{P}(E)}(-1)).
\end{equation*}

Thus we are only to verify:
\begin{equation*}
c(E) \cap [X] - c(F) \cap [X] = -c(E) \cap r_{*}\Big(\big(c(\mathscr{O}_{\mathbf{P}(E)}(-1))\big)^{-1} \cap[\mathbf{P}(C)]\Big).
\end{equation*}

Recall that the Segre class $s(E)$ of the vector bundle $E$ is the multiplicative inverse of the total Chern class $c(E)$ of $E$ (\cite{MR732620} chapter 3). We multiply both sides of the above formula by $s(E)$ and reach another equivalent form of formula \eqref{formula}:
\begin{equation*}
[X] - s(E) \cdot c(F) \cap [X] = -r_{*}\Big(\big(c(\mathscr{O}_{\mathbf{P}(E)}(-1))\big)^{-1} \cap[\mathbf{P}(C)]\Big).
\end{equation*}

Let us verify this last formula. Denote by $H$ a general hyperplane in $\mathbf{P}(E)$, or equivalently $c_1(\mathscr{O}_{\mathbf{P}(E)}(1))$. We have:
\begin{equation*}
\begin{split}
-r_*\Big(\big(c(\mathscr{O}_{\mathbf{P}(E)}(-1))\big)^{-1} \cap[\mathbf{P}(C)]\Big) &= -r_*\big(\sum_{i \geq 0}H^i \cap [\mathbf{P}(C)]\big) \\
                                                                                                                 &= -r_*\big(\sum_{i \geq 0}H^i \cdot c_n(r^*F \otimes \mathscr{O}(1)) \cap [\mathbf{P}(E)]\big) \\
                                                                                                                 &= -r_*\big(\sum_{i \geq 0}H^i \cdot \sum^{n}_{j=0}(c_{j}(r^*F) \cdot H^{n-j}) \cap [\mathbf{P}(E)] \big) \\
                                                                                                                 &= -\sum_{i \geq 0}\sum^{n}_{j=0}c_{j}(F) \cdot r_*(H^{n+i-j} \cap [\mathbf{P}(E)]) \\
                                                                                                                 &= -\sum_{i \geq 0}\sum^{n}_{j=0}c_{j}(F) \cdot s_{i-j+1}(E) \cap[X] \\
                                                                                                                 &= -\sum_{i \geq 0}\sum_{\substack{j+k=i+1 \\ j,k \geq 0}}c_j(F) \cdot s_k(E) \cap [X] \\
                                                                                                                 &= -\big(c(F) \cdot s(E) - 1\big) \cap [X] \\
                                                                                                                 &= [X] - c(F) \cdot s(E) \cap [X]
\end{split}
\end{equation*}

The second among these equalities uses Corollary \ref{imbed}. The fourth one uses the projection formula. The fifth one uses the definition of the Segre classes. The sixth one employs the fact that $s_k(E) = 0$ when $k<0$ for any vector bundle $E$. For the seventh one, recall that $c_0(F) = s_0(E) = 1$.

\end{proof}

Let us once again review the key ideas in the proof. We first take the dual of the proposed formula \eqref{formula} to make it more adaptable to the conclusion of theorem \ref{shadow}. Then the original formula is turned into a formula about the shadow of the quasi-symmetric blow-up $\mathbf{P}(C)$. The fact that the normal bundle to $\mathbf{P}(E)$ in $\mathbf{P}(E \oplus F)$ (and thus the normal bundle to $\mathbf{P}(C)$ in $\mathbf{P}(E)$) is related to the pull back of the logarithmic cotangent bundle $F$ is the most important observation in this paper. This observation finally allows one to express the shadow of $\mathbf{P}(C)$ by the Chern class of $E$ and the Segre class of $F$.

\bibliographystyle{alpha}
\bibliography{liaobib}

\begin{thebibliography}{CJNMM96}

\bibitem[Alu99]{MR1697199}
Paolo Aluffi.
\newblock Chern classes for singular hypersurfaces.
\newblock {\em Trans. Amer. Math. Soc.}, 351(10):3989--4026, 1999.

\bibitem[Alu04]{MR2097164}
Paolo Aluffi.
\newblock Shadows of blow-up algebras.
\newblock {\em Tohoku Math. J. (2)}, 56(4):593--619, 2004.

\bibitem[Alu12a]{aluffi}
Paolo Aluffi.
\newblock Chern classes of free hypersurface arrangements.
\newblock {\em Journal of Singularities}, 5:19---32, 2012.
\newblock arXiv:1201.5396.

\bibitem[Alu12b]{hyparr}
Paolo Aluffi.
\newblock Grothendieck classes and {C}hern classes of hyperplane arrangements.
\newblock {\em Int. Math. Res. Not.}, 2012.

\bibitem[CJNMM96]{MR1363009}
Francisco~J. Castro-Jim{\'e}nez, Luis Narv{\'a}ez-Macarro, and David Mond.
\newblock Cohomology of the complement of a free divisor.
\newblock {\em Trans. Amer. Math. Soc.}, 348(8):3037--3049, 1996.

\bibitem[CMNM02]{MR1931962}
Francisco Calder{\'o}n-Moreno and Luis Narv{\'a}ez-Macarro.
\newblock The module {$\mathscr{D}f^s$} for locally quasi-homogeneous free
  divisors.
\newblock {\em Compositio Math.}, 134(1):59--74, 2002.

\bibitem[CMNM09]{MR2500865}
F.~J. Calder{\'o}n~Moreno and L.~Narv{\'a}ez~Macarro.
\newblock On the logarithmic comparison theorem for integrable logarithmic
  connections.
\newblock {\em Proc. Lond. Math. Soc. (3)}, 98(3):585--606, 2009.

\bibitem[Eis95]{MR1322960}
David Eisenbud.
\newblock {\em Commutative algebra}, volume 150 of {\em Graduate Texts in
  Mathematics}.
\newblock Springer-Verlag, New York, 1995.
\newblock With a view toward algebraic geometry.

\bibitem[Ful84]{MR732620}
William Fulton.
\newblock {\em Intersection theory}, volume~2 of {\em Ergebnisse der Mathematik
  und ihrer Grenzgebiete (3) [Results in Mathematics and Related Areas (3)]}.
\newblock Springer-Verlag, Berlin, 1984.

\bibitem[Ken90]{MR1063344}
Gary Kennedy.
\newblock Mac{P}herson's {C}hern classes of singular algebraic varieties.
\newblock {\em Comm. Algebra}, 18(9):2821--2839, 1990.

\bibitem[Lia]{liao2}
Xia Liao.
\newblock Chern classes of logarithmic vector fields for locally-homogenous
  free divisors.
\newblock arXiv:1205.3843.

\bibitem[Lia12]{liao1}
Xia Liao.
\newblock Chern classes for logarithmic vector fields.
\newblock {\em Journal of Singularities}, 5:109---114, 2012.
\newblock arXiv:1201.6110.

\bibitem[Mac74]{MR0361141}
R.~D. MacPherson.
\newblock Chern classes for singular algebraic varieties.
\newblock {\em Ann. of Math. (2)}, 100:423--432, 1974.

\bibitem[Mat89]{MR1011461}
Hideyuki Matsumura.
\newblock {\em Commutative ring theory}, volume~8 of {\em Cambridge Studies in
  Advanced Mathematics}.
\newblock Cambridge University Press, Cambridge, second edition, 1989.
\newblock Translated from the Japanese by M. Reid.

\bibitem[MS01]{MR1843320}
Mircea Musta{\c{t}}{\v{a}} and Henry~K. Schenck.
\newblock The module of logarithmic {$p$}-forms of a locally free arrangement.
\newblock {\em J. Algebra}, 241(2):699--719, 2001.

\bibitem[NM]{macarro}
L.~Narv{\'a}ez~Macarro.
\newblock A duality approach to the symmetry of {Bernstein-Sato} polynomials of
  free divisors.
\newblock arXiv:1201.3594.

\bibitem[Sai80]{MR586450}
Kyoji Saito.
\newblock Theory of logarithmic differential forms and logarithmic vector
  fields.
\newblock {\em J. Fac. Sci. Univ. Tokyo Sect. IA Math.}, 27(2):265--291, 1980.

\end{thebibliography}

\end{document}